\documentclass[10pt,
twoside]{amsart}
\usepackage{lmodern} 
\usepackage{soul, wrapfig}
\usepackage{graphicx}
\usepackage{color}
\usepackage{amsmath}
\usepackage{amssymb}
\usepackage{amsfonts}
\usepackage{latexsym, amsthm, mathrsfs}
\usepackage{hyperref, breakurl}
\usepackage[OT1]{fontenc}
\usepackage{fancybox}
\usepackage{amscd}
\usepackage{enumitem}
\usepackage{fontenc}
\usepackage{amsthm}
\usepackage{graphicx}
\usepackage[english]{babel}
\usepackage{tikz-qtree}
\usepackage{rotating}
\usepackage{stmaryrd}

\newenvironment{myrules}{\begin{list}{}
{\setlength{\leftmargin}{0.8cm}
 \setlength{\labelwidth}{0.8cm}
 \setlength{\labelsep}{0.2cm}
 \setlength{\parsep}{0.3ex plus 0.2ex minus 0.1 ex}
 \setlength{\itemsep}{0.3ex plus 0.2 ex minus 0ex}
}}{\end{list}}

\numberwithin{equation}{section}
\newtheorem{theorem}{Theorem}[section]
\newtheorem{lemma}[theorem]{Lemma}

\theoremstyle{plain} 
\newtheorem{question}[theorem]{Question}

\newtheorem{proposition}[theorem]{Proposition}
\newtheorem{corollary}[theorem]{Corollary}

\newtheorem{claim}[theorem]{Claim}

\newtheorem*{maintheorem*}{Main Theorem}
\newtheorem*{conjecture*}{Conjecture}
\newtheorem*{theorem*}{Theorem}
\newtheorem*{proposition*}{Proposition}
\newtheorem*{corollary*}{Corollary}

\theoremstyle{definition} 
\newtheorem{definition}[theorem]{Definition}

\theoremstyle{remark}  
\newtheorem{remark}[theorem]{Remark}
\newtheorem{example}[theorem]{Example}
\newtheorem*{remarks*}{Remarks}
\newtheorem*{remark*}{Remark}
\newtheorem*{claim*}{Claim}

\renewcommand{\phi}{\varphi}

\newcommand{\algebra}{\textsc}

\newcommand{\amal}{\textsf{Am}}
\newcommand{\amoeba}{\poset{A}}

\newcommand{\baire}{\omega^\omega}

\newcommand{\cantor}{2^\omega}

\newcommand{\cohen}{\poset{C}}
\newcommand{\bC}{\poset{C}}
\newcommand{\conc}{\smallfrown}
\newcommand{\concat}{{}^\smallfrown} 

\newcommand{\enfa}{\textit}

\newcommand{\ideal}{\mathcal}

\newcommand{\ifif}{\Leftrightarrow}

\newcommand{\fspl}{\mathbb{FSP}}
\newcommand{\force}{\Vdash}

\newcommand{\laver}{\poset{L}}

\newcommand{\meager}{\ideal{M}}

\newcommand{\miller}{\poset{M}}

\newcommand{\model}{\textsc}

\newcommand{\poset}{\mathbb}

\newcommand{\real}{\baire}

\newcommand{\restric}{{\upharpoonright}}
\newcommand{\rest}{{\upharpoonright}}

\newcommand{\sacks}{\poset{S}}

\newcommand{\silver}{\poset{V}}
\newcommand{\mathias}{\poset{MA}}

\newcommand{\spl}{\poset{SP}}
\newcommand{\such}{\: : \:}

\newcommand{\nc}{\newcommand}
\nc{\marginparr}[1]
{\marginpar{\makebox[4mm]{} {#1}}}

\nc{\bP}{\mathbb P}
\nc{\nothing}[1]{}
\nc{\comment}[1]{#1}

\nc{\nco}{\DeclareMathOperator}
\nco{\codes}{code}
\nco{\taille}{taille}
\nco{\ter}{ter}
\nco{\rk}{rk}
\nco{\llower}{lower}
\nco{\order}{o}
\nco{\Nm}{Nm}
\nco{\ppower}{pp}
\nco{\pcf}{pcf} 
\nco{\tcf}{tcf} 
\nco{\tlim}{tlim} 
\nco{\limtext}{lim} 
\nco{\prodt}{{\textstyle \prod}}
\nco{\symdiff}{\triangle}
\nco{\dom}{dom}
\nco{\card}{card}
\nco{\lh}{lh}
\nco{\lt}{lt}
\nco{\lgg}{lg}
\nco{\hgt}{ht}
\nco{\rge}{range}
\nco{\otp}{otp} 
\nco{\trunk}{tr}
\nco{\nex}{next}
\nco{\reduction}{red}
\nco{\supt}{supt}
\nco{\supp}{supp}
\nco{\Lim}{Lim}
\nco{\Leb}{Leb}
\nco{\modd}{mod}
\nco{\invariant}{inv}
\nco{\id}{id}
\nco{\RO}{RO}
\nco{\poss}{pos}
\nco{\Inc}{Inc} 
\nco{\Fn}{Fn}
\nco{\add}{add}
\nco{\borel}{Bor}
\nco{\cof}{cof}
\nco{\cov}{cov}
\nco{\height}{ht}
\nco{\level}{Lev}
\nco{\levy}{Coll}
\nco{\non}{non}
\nco{\ot}{ot}
\nco{\rank}{Rank}
\nco{\splitting}{Split}
\nco{\splitlevel}{ns}
\nco{\stem}{stem}
\nco{\successor}{succ}
\nco{\Succ}{Suc}
\nco{\splsuc}{splsuc}
\nco{\Lev}{Lev}

\nc{\la}{\langle}
\nc{\ra}{\rangle}

\begin{document}
\title[On splitting trees
]{On splitting trees}
\author{Giorgio Laguzzi, Heike Mildenberger, Brendan Stuber-Rousselle} 
\date{January 27, 2020
}

\begin{abstract}
We investigate two variants of splitting tree forcing, their ideals and regularity properties. We prove connections with other well-known notions, such as Lebesgue measurablility, Baire- and Doughnut-property and the Marczewski field. Moreover, we  prove that any \emph{absolute} amoeba forcing for splitting trees necessarily adds a dominating real, providing more support to Spinas' and Hein's conjecture that $\add(\ideal{I}_\spl) \leq \mathfrak{b}$.  
\end{abstract}
\maketitle

\section{Introduction}

Trees and their associated forcing notions have been a crucial ingredient in set theory of the reals, specifically in questions concerning cardinal characteristics and regularity properties. The most popular such forcings are certainly $\sacks$, $\miller$, $\silver$, $\laver$ and $\mathias$ (Sacks, Miller, Silver, Laver and Mathias), but also other notions have played an important role; among them, there is some tradition in studying tree-forcing adding an \emph{$\omega$-splitting real}\footnote{An $\omega$-splitting real is a real $x$ in the forcing extension such that for any set $\{r_n \such n < \omega\}$ in the ground model that contains infinite sets $r_n$, for each $n$ we have that $r_n \cap x $ and $r_n \cap x^c$ are both infinite.}. Spinas \cite{Sp2004} introduced splitting tree forcing $\spl$, which has recently been studied also in \cite{SpHein}. We also investigate another form of splitting-tree forcing (called $\fspl$, Definition \ref{def:fspl}) somehow related to Spinas' one.

A notion of an ideal of \emph{small sets} can be introduced when dealing with any tree-forcing notion, as specified in \ref{def:ideal-meas}. For any such ideal one can associate the common cardinal characteristics, namely the covering, the additivity, the cofinality and the uniformity number.
Furthermore, a notion of measurability (and weak-measurability) generalizing the well-known Lebesgue-measurability and Baire property can be established when dealing with any type of tree-forcing notion (Definition \ref{def:ideal-meas}). It is well-known that in Solovay model any subset of the real line is $\poset{P}$-measurable, for a large variety of tree-forcing notions $\bP$, including $\spl$ and $\fspl$.

In the analysis of the associated additivity numbers a crucial role is played by the so-called amoeba forcings, which are posets adding \emph{generic trees}, see Definition \ref{amoeba}. 
In our paper we address a question raised in \cite{SpHein} by Spinas and Hein related to the additivity number of the ideal $\ideal{I_\spl}$ and the amoeba forcing for $\spl$. 
Even if we do not obtain a complete answer to the conjecture posed by Spinas and Hein, in Section \ref{S3}, Proposition \ref{2.4} and in Remark \ref{remark 3.8} we give more evidence to support such a conjecture, by showing that not only the natural amoeba for $\spl$ adds a dominating real, but that any \emph{aboslute} amoeba necessarily adds a dominating real. 

Section 4 contains a brief digression on Silver forcing, and connections between Silver-amobea and Cohen reals, in line with \cite{Sp15}.

In Section 5, we show some differences between the ideal $\ideal{I}_\fspl$ and the ideal of null sets. 

Our results in section 6 pertain only to the fat splitting forcing.
We show that for any $f$-slalom in the ground model there is an $\fspl$-name
that evades it. This is a strong negation of the Sacks property.

We prove in section 7 that actually the weak form of $\spl$- and $\fspl$-measurability for all sets of reals can be reached in a much simpler model, namely the $L(\mathbb{R})$ of the forcing extension obtained by a countable support iteration of Cohen forcing.
We conclude with an application of Shelah's amalgamation of forcing method
to fat splitting trees.
Using evasion for slaloms of width $n \mapsto 2^{kn}$, $k \geq 1$ we separate the regularity properties of $\fspl$ from others.

In the remainder of this introduction, we set up our notation and end with
one property of fat splitting.

\begin{definition}\label{1.1}
  \hfill
 \begin{myrules}
 \item[(a)] Let $X$ be a non-empty set. We let $X^{<\omega}=\{s\such (\exists n \in \omega)(s \colon n \to X)\}$.
The set $X^{<\omega}$ is partially ordered by the initial segment relation $\trianglelefteq$, namely  $s\trianglelefteq t$ if $s = t \restric \dom(s)$. We use  $\triangleleft$ for the strict relation.
For $s \in X^{<\omega}$ we let $\dom(s) = |s|$ be its domain.

\item[(b)]
A set $p \subseteq 2^{<\omega}$ is called a \emph{tree} if it is closed under initial
  segments, i.e. $t \in p \wedge s \trianglelefteq t \rightarrow s \in p$.
The elements of $p$ are called \emph{nodes}.
\item[(c)] For a node $s \in p$ we let $\Succ_p(s) = \{t \in p \such
  (s \triangleleft t \wedge |t| = |s| +1)\}$ be the set of immediate successor nodes.
  A node is called a \emph{splitting node} if it has two immediate successors in $p$.
  \item[(d)]  A tree p is called \emph{perfect}
  if for every $s \in p$ there is a splitting node $t \trianglerighteq s$.
\item[(e)] We write  $\splitting(p)$ for the set of splitting nodes of $p$.
\item[(f)] For $t \in p$, we write $\splsuc(t)$ for the shortest splitting node extending $t$.
  When $t$ is splitting, then $\splsuc(t)=t$.
\item[(g)] The \emph{stem of $p$}, short $\stem(p)$, is the $\trianglelefteq$-least splitting node of $p$.
\item[(h)] We define the splitting degree $\ot^p \colon \splitting(p) \rightarrow \omega$ recursively as follows:
\begin{myrules}
\item[-] $\ot^p(\stem(p))=0$
\item[-] for every $i \in \{ 0,1 \}$ and $t \in \splitting(p)$ with $\ot^p(t)=n$, put \\
  $\ot^p(\splsuc(t^\conc i))= n+1$.
\end{myrules}
\item[(i)]  For $n\in \omega $ let $\splitting_n(p)=\{t\in \splitting(p) \such \ot^p(t)=n \}$ and $\splitting_{\leq n}(p)=\{t\in \splitting(p) \such \ot^p(t)\leq n \}$.
\item[(j)] For $n \in \omega$, let $\Lev_n(p) := \{ t \in p\such |t| = n \}$.
\item[(k)] For $t \in p$ we let $p \restric t = \{s \in p \such s\trianglelefteq t \vee t \triangleleft s\}$.  
\item[(l)] For $F \subseteq p$ we let $p \restric F = \{s \in p \such
  (\exists t \in F)(s\trianglelefteq t \vee t \triangleleft s)\}$. 
\item[(m)] For $F \subseteq p$ we let $\ter(F) = \{s \in F \such
  (\neg\exists t \in F)(s\triangleleft t)\}$.  This is the set of terminal nodes of $F$.
\item[(n)] For each perfect tree $p \subseteq 2^{<\omega}$ we have
  a canonical splitting and lexicographical order preserving homomorphism
  \[h \colon \splitting(p)  \to 2^{<\omega}\]
  that is defined by induction on $n$ for arguments in $\splitting_n(p)$ as follows
  \begin{equation*}\begin{split}
      h(\stem(p)) = & \emptyset
      \\
      h(\splsuc_p(t ^\conc i)) = & h(t) ^\conc i \mbox{ for } t \in \splitting(p), i\in 2.
    \end{split}
  \end{equation*}
\item[(o)] We let $\bar{H}$ denote $h^{-1}$ and we let its lifting
    $H \colon 2^\omega \to [p]$ be defined by $H(f) = \bigcup\{ \bar{H}(f\restric n)\such n < \omega \}$.
\item[(p)] The body or rump of a tree $p \subseteq 2^{<\omega}$, short $[p]$, is
  the set $\{f \in 2^\omega \such (\forall n)( f \restric n \in p)\}$.

 \end{myrules}
\end{definition}

Spinas \cite{Sp2004} introduced splitting trees in order to analyse analytic splitting families.

\begin{definition} \label{def:spl}
  A tree $p \subseteq 2^{<\omega}$ is called \emph{splitting tree}, short $p \in \spl$, if for every $t \in p$ there is $k \in \omega$ such that for every $n \geq k$ and every $i \in \{0,1 \}$ there is $t' \in p$,
  $t \trianglelefteq t'$ such that $t'(n)=i$.
  We denote the smallest such $k$ by $K_p(t)$. The set $\spl$ is partially ordered by $q \leq_{\spl} p$ iff $q \subseteq p$.
\end{definition}

Of course, every splitting tree is perfect.

\nothing{
Spinas gave an alternative characterisation of a splitting tree:

\begin{definition}\label{seqspl}
  We let $\Fn(\omega,2)$ be the set of functions from a finite (not necessarily initial) subset of $\omega$ to $2$. We use Latin letters for elements of
  $2^{<\omega}$ and Greek letters for elements of $\Fn(\omega,2)$.
For $\sigma \in \Fn(\omega,2)$ we let $\stem(\sigma) = \sigma \restric \min(\omega \setminus \dom(\sigma))$.
  
For a tree $T \subseteq 2^{<\omega}$, we let $\bar{T} = \{\sigma \in \Fn(\omega,2) \such \exists t \in T \such \sigma \subseteq t\}$.
  
Given $\tau \in \bar{T}$ the pair $(\tau_0,\tau_1) \in \bar{T} \times \bar{T}$ is called \emph{a splitting pair in $T$  above $\tau$} if
\begin{myrules}
\item[(a)] $\tau_0,\tau_1 \supseteq \tau$,
\item[(b)] $\stem(\tau) \in \dom(\tau_0) \cap \dom(\tau_1)$ and 
\item[(c)] $  (\forall j \in \dom(\tau_0) \cap \dom(\tau_1) \setminus \dom(\tau))( \tau_0(j) \neq \tau_1(j))$.
\end{myrules}
A perfect tree $p \subseteq 2^{<\omega}$ is called a \emph{sequentially splitting tree} if $p$
has a splitting generating sequence, i.e., if there are
$\la \phi_s \such s \in 2^{<\omega}\ra$ such that
for
every $s \in 2^{<\omega}$, $(\phi_{s^\conc 0}, \phi_{s^\conc 1})$ is a splitting pair above $\phi_s$,
and $p$ 
is the closure of $\{ \stem(\phi_s) \such  s \in 2^{<\omega} \}$ under initial segments.
\end{definition}

\begin{theorem}\cite{Sp2004} Every splitting tree $p$ is sequentially splitting.
\end{theorem}

\begin{proof} For completeness and understanding, we write a proof.
  
\begin{definition} For $t \in p \in \spl$ we let
  \[\taille(t) =
  \{n \in \omega \such (\exists i_n\in 2) (\forall s \in p \rest t)( s(n)= i_n)\}.\]
\end{definition}

Since $p \in \spl$, for any $t \in p$ the set $\taille(t)$ is finite and
\[K_p(t) = \max(\taille(t)) +1.\]

By induction on $n$ we define for $s \in 2^n$ a function $\phi_s \in \bar{ p}$ such that $(\phi_{s \concat 0}, \phi_{s \concat 1})$ is a splitting pair above $\phi_s$ and in addition
\begin{equation}
  \label{ast}
  \phi_s = \bigl\{(m,i_m) \such (m \in \taille(\stem(\phi_s)) \wedge
  (\exists t \in p \rest \stem(\phi_s))(t(m)=i_m))\bigr\}.
\end{equation} 
The quantifier $(\exists t \in p \rest \stem(\phi_s))$ can equivalently be replaced by $(\forall t \in p \rest \stem(\phi_s))$.
We let 
\begin{equation*}\begin{split}
\phi_\emptyset = \{(m,i_m) \such (m \in \taille(\stem(p)))
\wedge
(\exists t \in p)(t(m) = i_m)\}.
  \end{split}
\end{equation*}
By definition of $\stem(p)$, we have $ \stem(\phi_\emptyset) = \stem(p)$.

Now suppose that $\phi_s$, $s \in 2^{\leq n}$ are such that for $ t \in 2^{<n}$,
$(\varphi_{t ^\conc 0},\varphi_{t ^\conc 1})$ is a splitting pair in $p$ above $\varphi_t$.
and such that Equation \eqref{ast} holds for $\varphi_{t^\conc i}$.
Then 
for $s \in 2^n$, we have that $\stem(\phi_s)$ is splitting in $p$ and therefore $|\stem(\phi_s)|  \not\in \taille(\phi_s)$.
We let
\begin{equation*}\begin{split}
\phi_{s^\conc i} =  \bigl\{(m,i_m) \such & (m \in \taille(\stem(\phi_s)^\conc i)) \wedge (\exists t \in p \restric (\stem(\phi_{s}) ^\conc i))(t(m) = i_m)\bigr\}.
\end{split}\end{equation*}
Then for $i\in 2$ we have
\begin{eqnarray*}
  \stem(\phi_{s^\conc i})) = \splsuc(\stem(\phi_s) \concat i)
  \end{eqnarray*}
Hence $\phi_{s\concat i}$ fulfills Equation~\eqref{ast} and by the properties of $\taille(\stem(\phi_s))$, the pair 
$(\phi_{s^\conc 0}, \phi_{s^\conc 1})$ is a splitting pair in $p$ above $\phi_s$.

Since any $t \in \Lev_n(p)$ is an initial segment of
$\stem(\phi_s)$ for some $s\in 2^{\leq n}$, we have that
the downward closure of $\{\stem(\phi_s) \such s \in 2^{<\omega}\}$
is $p$.
\end{proof}
}

Now we introduce a relative of splitting tree forcing that  has stronger
splitting properties. We do not know whether strictly stronger.

\begin{definition} \label{def:fspl}
A perfect tree $p \subseteq 2^{<\omega}$ is called \emph{a fat splitting tree} ($p \in \fspl$) iff for every $t \in p$ there is $k \in \omega$ such that for every $n\geq k-1$ there is $t'\in \Lev_{n}(p\restric t)$ such that $t'\in \splitting(p)$. We denote the smallest such $k$ by $K_p(t)$. Again subtrees are stronger, i.e., smaller,  conditions.
\end{definition}

\begin{definition} Let $X, S \in [\omega]^\omega$. We say \emph{$S$ splits $X$}
  if $S \cap X$ and $X \setminus S$ are both infinite.
  \end{definition}

\begin{proposition} The forcing $\fspl$ adds a generic real
  \[x_G = \bigcup\{\stem(p) \such p \in G\}\]
  such that for any infinite set $\{\{n_i, n_i+1\} \such i < \omega \}$
  in the ground model there are infinitely many $i$
  such that \[|x_G \cap \{n_i,n_i+1\}| = 1.\]
\end{proposition}

\begin{proof} We prove only the latter part. Let $p \in \fspl$ and $\{\{n_i, n_i+1\} \such i < \omega \}$
  and $k \in \omega$ be given. After possibly strengthening $p$ we can assume  $|\stem(p)| > k$. We take $i$ such that $n_i \geq K_p(\stem(p))-1\geq k$.
  Then there is $s \in \splitting(p)$ such that $|s| = n_i+1$.
  Assume that $s(n_i) = 0$. There is $s^\conc 1  \in \Succ_p(s) $.
  We let $q = p \rest s ^\conc 1$. The other case is symmetric.
  Since $k$ was arbitrary, we have  $p \Vdash (\exists^\infty i)
  (|x_G \cap \{n_i,n_i+1\}| = 1)$, as claimed.
\end{proof}

We do not know whether $\spl$ has the same property.

\begin{remark}
  \begin{myrules}
    \item [(1)]
Any fat splitting tree is a splitting tree, and the function $K_p$ in
the sense of the fat splitting
is an upper bound to a function witnessing splitting.
We use the same function symbol $K_p$ for the forcing orders $\spl$ and $\fspl$, although the interpretation of the symbol depends on the underlying forcing order.
The technical treatment of the $K_p$ is the same in both interpretations.

\item[(2)]
For $\bP \in \{\spl, \fspl\}$, $K_p$ in the respective meaning, $K_p(s) \leq K_p(t)$ for $s \trianglelefteq t$.
\item[(3)]
For $\bP \in \{\spl,\fspl\}$ we have that $K_p(t)= K_p(\splsuc(t))$ for every condition $p\in\bP$ and node $t\in p$.
  \end{myrules}
\end{remark}

\section{Axiom A and dividing a condition into two}

 We  provide some basic properties of $\fspl$ whose analogues for $\spl$ were
  proved by Hein and Spinas \cite{SpHein}.
 In the beginning of the section we show that $\fspl$ has strong Axiom A and that we can arrange lower bounds for the
  $K_p$ values.
We build on Spinas' and Hein's techniques for splitting trees and
  develop them further, both for $\spl$ and for $\fspl$.

\begin{definition}\label{Axiom_A}
  A notion of forcing $(\bP,\leq)$ has \emph{Axiom A} if there are partial order relations $\la \leq_n \such n < \omega \ra$ such that
  \begin{myrules}
  \item[(a)] $q \leq_{n+1} p$ implies $q\leq_n p$ , $q\leq_0 p$ implies
    $q \leq p$,
  \item[(b)] If $\la p_n \such n < \omega \ra$ is a fusion sequence, i.e.,  a sequence such that
    for any $n$, $p_{n+1} \leq_n p_n$, then there is a lower bound $p \in \bP$, $p \leq_n p_n$. 
  \item[(c)] For any maximal antichain $A$ in $\bP$ and and $n\in \omega$ and any $p \in \bP$ there is $q \leq_n p$ such that only countably many elements of $A$ are compatible with $q$. Equivalently, for any open dense set $D$ and any $n$, $p$, there is a countable set $E_{p}$ of conditions in $D$ and $q\leq_n p$ such that $E_{p}$ is predense below $q$.

    A notion of forcing $(\bP,\leq)$ has \emph{strong Axiom A} if the set of compatible elements in (c) is even finite. 
 \end{myrules}
\end{definition}

Axiom A entails properness and
strong Axiom A implies ${}^\omega \omega$-bounding (see, e.g., \cite[Theorem 2.1.4,
  Cor~2.1.12]{RoSh:470}).

  \begin{definition}\label{leqn}
    For $\bP = \fspl$,  we define a decreasing sequence of partial orderings $\langle \leq_n \such n\in \omega \rangle$ by $q \leq_n p$ if
\[q \leq p \land
(\splitting_{\leq n}(p) = \splitting_{\leq n}(q) \land (\forall t \in \splitting_{\leq n}(p)) (K_q(t)=K_p(t)). 
\]
\end{definition}

  \nothing{ Remark
    If $K_0 = 0$ and $K_{n+1} = \max\{K_p(t) \such t \in \Lev_{\leq K_n}(p)\} +1$, then  $(\Lev_{\leq n}(p) = \Lev_{\leq n}(q)$ and $(\forall t \in \Lev_{\leq K_j}(p)) K_p(t) = K_q(t)$ implies that $\splitting_j(p) = \splitting_j(q)$.
  }

  Spinas and Hein \cite[Lemma 3.9]{SpHein} introduce a countable
  notion of forcing:

\begin{definition}\label{bP_p}
Let $\bP  \in \{ \spl, \fspl\}$ and let $p \in \bP$. We define $\bP_p$:
  Conditions in $\bP_p$ are finite trees $F \subseteq p$ such that there is $g_F \in \omega$ such that $\ter(F) \subseteq 2^{g_F}$.

  We let $F' \leq_{\bP_p} F$ if
  \begin{myrules}
  \item[(a)]
    $F' \supseteq F$ and
  \item[(b)]
    $\forall s \in F$, $K_{p \rest F'}(s) =
    K_{p \rest F}(s)$.
    \end{myrules}
\end{definition}

By Lemma \ref{triple}, the forcing $\bP$ is atomless.
Hence also $\bP_p$ is atomless and equivalent to Cohen forcing.
The union of a $\bP_p$-generic condition is a condition $p_G$ in $\bP$ again.
\nothing{Not only from $p_G$ as a whole but also of each branch of $p_G$ one
can read off $G$ after a certain preparation. The following lemma
is a step towards a proof of the latter statement.
\begin{lemma}\label{embed}
  Let $p \in \fspl$. Then there is $q \leq_0 p$ and there is a  $\trianglelefteq$-preserving $h_q \colon q \to 2^{<\omega}$ such that
  for any $s \in 2^{<\omega}$, $q \rest h_q^{-1}[\{s\}] \leq_{0} p$ for any $s, t \in 2^{<\omega}$
  \begin{equation}\label{embed_eins}
  s \perp t \leftrightarrow q \rest h_q^{-1}[\{s\}] \perp q \rest h_q^{-1}[\{t\}]
    \end{equation}
  and 
  \begin{equation}\label{embed_zwei}
    s \trianglelefteq t \leftrightarrow q \rest h_q^{-1}[\{s\}] \geq_{|s|} q \rest h_q^{-1}[\{t\}]
  \end{equation}
$h$ is far from being injective. 
   Suppose now we extend the universe by Cohen forcing. Then any Cohen-generic real $c$ can be read off as
   from the fusion sequence $(q \rest h^{-1}(c\rest n))_n$, and each branch $b$ of
   the fusion limit of the latter gives $\bigcup_n h(b \rest n) = c$. 
\end{lemma}
}
Our
Lemma~\ref{triple}
is based on Spinas and Hein's proof of \cite[Lemma 3.9]{SpHein}.

\begin{lemma}\label{triple} 
  Let $k \in \omega$,  $p \in \fspl$, $m \in \omega$, $D$ open dense in $\fspl$.
  Then for  $t \in \splitting_k(p)$, $j = 0,1$ there is $p_{t,j}$ and a finite set $E_{t,j}$ with the following properties:
 \begin{myrules}
  \item[(1)]
    \begin{align*}
  &      p_{t,j} \leq_0 p \restric t  \wedge  E_{t,j} \subseteq D \wedge
        E_{t,j} \mbox{ is predense below } p_{t,j} \wedge
        \\
        & K_{p_{t,j}}(t) = K_p(t) \wedge \bigwedge_{j=0,1} K_{p_{t,j}}(t^\conc 1) > K_{p_{t,j}}(t^\conc 0) \geq m
    \end{align*}
 
\item[(2)]
    For each $j = 0,1$ separately, the
    \[\{K_{p_{j}}(\splsuc(t^\conc i))  \such t \in
    \splitting_k(p_j)= \splitting_k(p), i = 0,1\}
    \]
    are ordered lexicographically according to the splitting preserving
    homomorphic image of
    $\{t^\conc i \such t \in \splitting_k(p), i = 0,1\}$ in $(2^{k+1}, \leq_{\rm lex})$.
    
    Moreover, for each $t \in \splitting_k(p_j)$, $i = 0,1$,
    $K_{p_{j}}(\splsuc(t^\conc i))$ is strictly larger than $\max\{K_{p_{j}}(s)\such s \in \splitting_{\leq k}(p_j)\}$.

\item[(3)]
    For $j = 0,1$ we let
\[p_j = \bigcup \{p_{t,j} \such t \in \splitting_k(p)\}.\]
Then we have $p_j \leq_k p$ 
and $p_0 \perp p_1$ (even $p_0 \cap p_1$ is finite) and $p_0 \cup p_1 \leq_k p$.
There is a finite set of strengthenings of $p_j$
that is a subset of $D$ and predense below $p_j$.
\nothing{
There is  $h_q \supseteq h$ that describes the incompatibility structure
from Equation~\ref{embed_eins} and Equation~\ref{embed_zwei}.
so that it serves as one step in Lemma~\ref{embed}.}\end{myrules}

\end{lemma}
  
\begin{proof}  We recall, $h$ is defined in Def. \ref{1.1}(n). We go along the lexicographic order $\leq_{\rm lex}$ of $h(t)$
  for $t \in \splitting_k(p)$. Suppose the construction has already be performed for $t' \in \splitting_k(p)$ such that $h(t') \leq_{\rm lex} h(t)$.

  We assume that $m >
 \max\{K_{p_{j}}(s)\such s \in \splitting_{\leq k}(p_j)\}$.

 \emph{Thinning out above $t ^\conc 0$:}\\
  Let
  \[
  K_0^t = \max(\{K_q(t'^\conc 0), K_q(t'^\conc 1) \such t' \in \splitting_k(p), h(t') \leq_{\rm lex} h(t)\} \cup \{m\}).
  \]
For any $s \in \Lev_{K_0^t}(p \restric t ^\conc 0)$, $j \in 2$, there is $p_{s,j}$ such that
\begin{equation}
  \label{j_eins}
  p_{s,j} \leq p \restric (\splsuc(s) ^\conc j) \wedge p_{s,j} \in D.
    \end{equation}
Note that all the $p_{s,j}$ contain $\splsuc(s) ^\conc j$ and do not contain $\splsuc(s) ^\conc (1-j)$.

\emph{Thinning out above $t ^\conc 1$:}\\    
    Let
    \[K^t_1 = \max\{K_{p_{s,j}}( \splsuc(s) \concat j) \such s \in \Lev_{K_0^t}(p \restric t ^\conc 0), j \in 2\} + 1.
    \]
    
    For any $s \in \Lev_{K^t_1}(p \restric t ^\conc 1)$ there is $p_{s,j}$ such that
    \begin{equation}
      \label{j_zwei}
      p_{s,j} \leq p \restric (\splsuc(s) ^\conc j) \wedge p_{s,j} \in D.
    \end{equation}
    Again all the $p_{s,j}$ contain $\splsuc(s) ^\conc j$ and do not contain $\splsuc(s) ^\conc (1-j)$.    
 We let
    \[p_{t,j} = \bigcup \{p_{s,j} \such  s \in \Lev_{K^t_0}(p \restric (t ^\conc 0))\}
    \cup  \bigcup \{p_{s,j} \such s \in \Lev_{K^t_1}(p \restric (t ^\conc 1))\}.
    \]
For $j = 0,1$,    
a finite subset $E_{t,j}$ of $D$ that is predense below $p_{t,j}$ is
\begin{equation*}
  \begin{split}
E_{t,j}   =   \{p_{s,j} \such s \in \Lev_{K^t_0}(p \restric (t ^\conc 0))\}
 \cup  \{p_{s,j} \such s \in \Lev_{K^t_1}(p \rest (t ^\conc 1))\}.
  \end{split}
\end{equation*}
Then $p_j \restric (t ^\conc 0)$ (by reserving enough splitting in the interval above $K^t_1$) and $p_j \restric (t^\conc 1)$ (for the interval below $K^t_1$) together witness
    that we have $K_{p_{t,j}}(t) = K_p(t)$.
 As stated, we let $p_j = \bigcup\{p_{t,j} \, : \, t \in \Lev_k(p)\}$.
Thus we have $p_j \leq_k p$.
    By construction, for any $j =0,1$, $t \in \splitting_k(p)$,
\[K_{p_{t,j}}(\splsuc(t ^\conc 1)) \geq  K^t_1 > K_{p_{t,j}}(\splsuc(t ^\conc 0)) \geq K^t_0 \geq m\] and the lexicographic order is carried on.

We turn to property (3).
A finite subset of $D$ that is predense below $p_j$ is given by $\bigcup \{E_{t,j} \such t \in \splitting_k(p)\}$.
    Finally, properties \eqref{j_eins} and \eqref{j_zwei} guarantee
    $p_0 \perp p_1$.
  \end{proof} 

\begin{corollary} For any $p \in \fspl$ there are $2^\omega$ mutually incompatible
  conditions stronger than $p$.
\end{corollary}

\begin{proof} By in induction on $\dom(s)$ we construct for
  $j = 0,1$, $p_{s^\conc j} \leq_{|s|} p_s$.
  The successor step is like the previous lemma with $p_s = p$
  and $p_{s ^\conc j} = p_j$ from there.
  Now that $p_s$ for $s \in 2^{<\omega}$ are defined, we let
  for $b \in 2^\omega$, $p_b  = \bigcap \{p_{b \rest n} \such n < \omega\}$.
  Since $p_{b \rest n}$, $n < \omega$, is a fusion sequence, $p_b$ is a condition. \end{proof}

\begin{corollary}
  The set of fat splitting trees $p\in\fspl$ with the property
  such that there is a splitting preserving homomorphism $h$ from
  $\splitting(p) $ onto $2^{<\omega}$ such that for every
  $s,t$ such that $|h(s)| = |h(t)| $ and $h(s) \leq_{\rm lex} h(t)$ we have
  \[K_p(s) < K_p(t)\] is dense.
\end{corollary}
\begin{proof}
  Let $p\in \fspl$. We construct a fusion sequence $\langle p_n \such n\in \omega \rangle$ according to Lemma \ref{triple} by letting $j = 0$ all the time.
  \end{proof}
\nothing{
  Proof of lemma \ref{embed}:

  \begin{proof} Given $p$,
  by induction on $n$ we choose $\la q_n, k_n, h_n \such n < \omega\ra$
  such that \begin{myrules}
   \item[(1)] $q_0 = p$,
\item[(2)]   $ q_{n+1} \leq_n q_n$,
\item[(3)]   $q_{n+1} \rest k_n =  q_n$,
  \item[(4)]
    $h_n \colon q_n \rest \Lev_{k_n} (q_n) \to 2^{<\omega}$
    has Equation~\eqref{embed_eins} and Equation~\eqref{embed_zwei} for $s \in 2^{\leq n}$ 
  \item[(5)] $h_{n+1} \supseteq h_n$.
  \end{myrules}
  According to Lemma~\ref{triple} there is such a sequence.
  Then $q = \bigcap\{ q_n \such n < \omega\}$ and $h = \bigcup \{h_n \such n<\omega\}$ have the required properties.
  \end{proof}
  }
\begin{proposition}\label{strong_Axiom_A} The expanded fat splitting forcing
    $(\fspl,\leq, (\leq_n)_n)$ has strong  Axiom A. 
The same holds for  $\spl$.
\end{proposition}

\begin{proof}
The result for $\spl$ is already known by work of Shelah and Spinas (see \cite{Sp2004}). The fusion property follows from the definition of $\leq_n$.
By Lemma~\ref{triple}, the forcing order $\fspl$ together with $\leq_n$
according to Definition~\ref{leqn} has strong Axiom A.
\end{proof}

\section{Amoeba forcing and dominating reals}
\label{S3}
In the section we deal with a question addressed by Spinas and Hein, giving more evidence to support their conjecture.

\begin{definition} \label{def:ideal-meas}
  Let $\poset{P}$  be a forcing whose conditions
  are perfect trees ordered by inclusion,
  in particular $\poset{P}$ could be $\spl$, $\fspl$.
  \begin{myrules}
  \item[(1)]
    A subset $X \subseteq 2^\omega$ is called \emph{$\poset{P}$-nowhere dense,} 
if
\[
(\forall p \in \poset{P})( \exists q \leq p) ([q] \cap X = \emptyset).
\]
We denote the ideal of $\mathbb{P}$-nowhere dense sets with $\mathcal{N}_\mathbb{P}$.
\item[(2)]
  A subset $X \subseteq 2^\omega$ is called \emph{$\mathbb{P}$-meager} if it is included in a countable union of $\mathbb{P}$-nowhere dense sets. We denote the $\sigma$-ideal of $\poset{P}$-meager sets with ${\mathcal I}_{\poset{P}}$.
\item[(3)]
A subset $X \subseteq 2^\omega$ is called \emph{ $\poset{P}$-measurable} if 
\[
(\forall p \in \poset{P})( \exists q \leq p) ([q]\setminus X \in \mathcal{I}_\mathbb{P} \vee [q] \cap X \in \mathcal{I}_\mathbb{P}).
\]
\item[(4)]
A subset $X \subseteq 2^\omega$ is called \emph{weakly $\poset{P}$-measurable} if 
\[
(\exists q)([q]\setminus X \in \mathcal{I}_\mathbb{P} \vee [q] \cap X \in \mathcal{I}_\mathbb{P}).
\]
\end{myrules}
\end{definition}
Notice that these notions generalize some well-known ones:
\vspace{3mm}
\begin{center}
\begin{tabular}{|c|c|c|}\hline
 $\mathbb{P}$ & $\mathcal{I_\mathbb{P}}$  & $\mathbb{P}$-measurable \\ \hline
$\mathbb{C}$ & meager ideal  & Baire property \\ \hline
$\mathbb{B}$ & null ideal  &  measurable \\ \hline
$\mathbb{V}$ & Doughnut null  & Doughnut-property \cite{Prisco2000} \\ \hline
$\mathbb{S}$ & Marczewski ideal \cite{Szpilrajn1935}  & Marczewski field \\ \hline
\end{tabular}
\end{center}

\begin{remark}
In general the two ideals do not coincide $\mathcal{I}_\mathbb{P}\neq\mathcal{N}_\mathbb{P}$, for instance when $\mathbb{P}$ is the Cohen forcing. In many other cases however, they do coincide  $\mathcal{I}_\mathbb{P}=\mathcal{N}_\mathbb{P}$, for instance when $\mathbb{P}\in \{\mathbb{S},\mathbb{L},\mathbb{M},\mathbb{V},\mathbb{MA} \}$. When the latter occurs $X$ is $\mathbb{P}$-measurable if and only if:
\[
(\forall p \in \poset{P})( \exists q \leq p) ([q]\subseteq  X  \vee [q] \cap X = \emptyset ).
\]
In our specific case for $\mathbb{P}\in \{\spl,\fspl \}$  the ideal of $\mathbb{P}$-nowhere dense sets is in fact a $\sigma$-ideal and so $\mathcal{I}_\mathbb{P}=\mathcal{N}_\mathbb{P}$. The proof is a fusion argument and is a consequence of Lemma \ref{triple}. In fact, let $X\subseteq\bigcup_n X_n$ with $X_n \in \mathcal{N}_\mathbb{P}$, for $n\in \omega$. Fix a condition $p\in \mathbb{P}$ and recursively apply Lemma \ref{triple} with $D=2^\omega \setminus X_n$ in order to construct a fusion sequence $p=p_0 \geq_0 p_1 \geq_1 \dots$ with the property that for every $n\in \omega$, $[p_n]\cap X_n= \emptyset$. And so, the fusion $q=\bigcap_n p_n$ is such that $[q]\cap X=\emptyset$. 
\end{remark}

\begin{remark} \label{Remark2}
Weak-$\poset{P}$-measurability is a weaker statement then $\poset{P}$-measurability, and if referred to a single set, it is not even a regularity property, as a given set can contain the branches through a tree $p \in \poset{P}$ but being very irregular outside of $[p]$. However classwise staments about weak measurability are in some cases sufficient to obtain measurability. More precisely, let $\Gamma$ be a family of subsets of reals and
\[
\begin{split}
\Gamma(\poset{P})&:= \text{``all sets in $\Gamma$ are $\poset{P}$-measurable''} \\
\Gamma_w(\poset{P})&:= \text{``all sets in $\Gamma$ are $\poset{P}$-measurable''}.
\end{split}
\]
If $\Gamma$ is closed under continuous pre-image and intersection with closed sets, for $\poset{P} \in \{ \sacks, \silver, \miller,\laver, \mathias \}$ one has $\Gamma(\poset{P}) \ifif \Gamma_w(\poset{P})$ (see \cite[Lemma 2.1]{BL99} and \cite[Lemma 1.4]{BLH2005}). Hence, we can obtain some straightforward implications, such as $\Gamma(\laver) \Rightarrow \Gamma(\miller)$ and $\Gamma(\silver) \Rightarrow \Gamma(\sacks)$. 
\end{remark}

\begin{definition}
  (See  \cite[Definition 3.11]{SpHein})
  Let $\poset{P}$ be the splitting forcing or the fat splitting forcing,
  and let $K_p(t)$ be minimal with the properties in the definitions.
  We define $d \colon \poset{P} \to \omega^\omega$ as follows:
\[
d_p(n):=  \min \{ K_p(t^\conc i): i \in 2, t \in \splitting_{n}(p)\}.
\]
\end{definition}

\begin{lemma} \label{lemma:dominate}
Given $f \in \omega$ and $p \in \{ \spl, \fspl  \}$, there exists $q \leq p$ such that $f \leq^* d_q$.
\end{lemma}

A proof for $\spl$ can be found \cite[Lemma 3.14]{SpHein} and the argument for $\fspl$ is analogous and we leave the details to the reader.

The eventual domination relation is denoted by $\leq^*$,
and $\leq$ means sharp domination on $\omega^\omega$.
Note that for two (fat) splitting trees we have $q \leq_\poset{P} p$  implies $d_p \leq d_q$.

\begin{definition} \label{amoeba}
We say that $q$ is an \emph{absolute $\poset{P}$-generic tree over $V$} iff $q \in \poset{P}$ and all its branches are $\poset{P}$-generic over $V$ in any extension, more precisely such that for any ZF-model extension $N \supseteq V$ we have 
\[
N \models q \in \poset{P} \land \forall x \in [q] \text{($x$ is $\poset{P}$-generic over $V$)}.
\]
An \emph{absolute amoeba forcing for $\poset{P}$} is a poset adding an absolute $\poset{P}$-generic tree. 
\end{definition}
We remark that every amoeba forcing in literature, at least to our knowledge, satisfies this property including the natural amoeba for $\spl$ as defined in (see \cite[Definition 3.15.]{SpHein}). Other examples would be the versions of amoeba for Laver and Miller (see \cite[pp 709 and 714]{Sp95}) as proven in \cite[Lemma 1.1.7, Lemma 1.1.8, Remark 1.1.10]{Sp95}.\\
The main idea of using an amoeba forcing is to add a \emph{large} set of \emph{generic} reals. However, we must be careful that this notion is sufficiently absolute, otherwise we might end up with a  useless amoeba. For example, if $G$ is a Sacks-generic filter over $V$, then it is well-known that  in $V[G]$ there is a perfect set  $P$ of Sacks-generic reals. But if we take $H$ a Sacks-generic filter over $V[G]$, then in $V[G][H]$ the set $P$ is no longer a perfect set of Sacks-generic reals. Moreover, in $V[G][H]$ the set of Sacks-generic reals over $V$ is in the Marczewski ideal $\mathcal{I_\mathbb{S}}$.\\
The known application of an amoeba forcing to blow up the additivity number like in \cite[Theorem 1.3.1]{Sp95} uses an absoluteness argument which justifies Definition \ref{amoeba}.\\
In light of that, the following proposition provides more support to Spinas' and Hein's conjecture that  $\add(\ideal{I}_\spl) \leq \mathfrak{b}$ and that any reasonable amoeba for $\spl$ adds a dominating real.


\begin{proposition}\label{2.4}
Let $\poset{P} \in \{ \spl, \fspl \}$ and let $V \subseteq N$ be models of ZFC. If
\[
N \models \text{``There is an absolute $\poset{P}$-generic tree over $V$''}
\]
Then
\[
N \models \text{``There is a dominating real over $V$''}. 
\]
Hence any absolute amoeba forcing for $\poset{P}$ adds a dominating real. 
\end{proposition}

\begin{proof}
 We construct an $\omega$ sized family $\{q_k \such k\in \omega\}$ that dominates all $f \in V \cap \omega^\omega$. Note that this is enough since we can use a standard diagonal argument and define $x\in \omega^\omega$ as
$$x(n):= \sup\{d_k (n) \such k\leq n \} +1 .$$
Then $x$ almost dominates all $f \in V \cap \omega^\omega$. 

Let $q \in N$ be the $\poset{P}$-generic tree over $V$.

Now let $\bar H: 2^{<\omega} \rightarrow \splitting(q)$ be as in Definition \ref{1.1} (o), i.e., $\bar H(\emptyset)=\stem(q)$ and $\bar H(\sigma^\conc i)=  \splsuc(\bar H(\sigma)^\conc i)$, for every $\sigma\ \in 2^{<\omega}$, $i \in \{ 0,1 \}$; and $H \colon  2^\omega \rightarrow [q]$ be its natural associated extension.  Then consider $\dot c$ the canonical $\cohen$-name for a Cohen real over $N$. Note that for every Cohen real $c$ over $N$ we have  
\[
N[c] \models H(c) \in [q] \land \text{$H(c)$ is $\poset{P}$-generic over $V$}.
\]
Let $G_c$ be the filter associated with $H(c)$ that is $\poset{P}$-generic over $V$, i.e., $N[c] \models H(c)= \bigcap \{[p]: p \in G_c\}$.  Hence for any $\poset{P}$-open dense set $D \in V$  there is $p \in D \cap G_c$, and so $N[c] \models H(c) \in[p]$. Hence there is $\sigma_{D,p} \in \cohen$ such that $\sigma_{D,p} \force  H(\dot c) \in [p]$ (in this case this follows since $H(c)$ is $\poset{P}$-generic over $V$ and therefore given any $\poset{P}$-open dense $D$ subset of $\poset{P}$ in $V$ the $\poset{P}$-generic $H(c)$ has to meet such $D$.)

 So we get the following:\\\
for every $\poset{P}$-open dense $D$ of $\poset{P}$ in $V$ there exists $ p\in D$ and $\sigma \in 2^{<\omega}$ such that
\begin{equation}\label{Cohen real}
  \begin{split}
&\text{for each Cohen real $c$ extending $\sigma$} \\
    &N[c]\models H({c})\in [p] \cap  [q\restric \bar {H}(\sigma)].
  \end{split}
\end{equation}
\begin{claim}\label{q H}
From (\ref{Cohen real}) it follows that $q\restric \bar{H}(\sigma) \leq p$. 
\end{claim}
To prove the claim we argue by contradiction, so assume $t\in q\restric \bar{H}(\sigma) \setminus p$. Pick $\tau \trianglerighteq \sigma$ such that $\bar{H}(\tau) \trianglerighteq t$. Now fix any Cohen real $c$ extending $\tau$. Then $H(c)$ extends $t$ and therefore $H(c)\in [q\restric \bar{H}(\sigma)] \setminus [p]$. This contradicts \ref{Cohen real}.\\

Now we finish the proof of Proposition~\ref{2.4}.
Let $\{q_k: k \in \omega\}$ enumerate $\{q\restric \bar{H}(\sigma) \such \sigma \in \cohen\}$. We claim that $\{d_{q_k} \such k\in \omega \}$ dominates all $f \in V \cap \omega^\omega$. In fact, fix   $f \in V \cap \omega^\omega$. By Lemma~\ref{lemma:dominate}, for every $f \in V \cap \omega^\omega$ we pick
a $\poset{P}$-open dense $D_f \subseteq \poset{P}$ in $V$ such that for every $p \in D_f$, $f\leq^* d_p$. Then pick $\sigma \in \cohen$ and $p\in D_f$ as in (\ref{Cohen real}). Take $k\in \omega$ such that $q_{k}=q\restric \bar{H}(\sigma)$. Note that by Claim \ref{q H} we have that $q_{k} \leq p$ and so for all but finitely many $n \in \omega$, 
$d_{q_{k}}(n) > f(n)$.

\end{proof}

\begin{remark} \label{remark 3.8}
Remind that the bounding number, $\mathfrak b$, is the minimal size of an $\leq^*$-unbounded
family. Beyond what we prove in Proposition \ref{2.4}, in \cite{SpHein} Spinas and Hein addressed also the following parallel question: is $\text{add}(\mathcal{I}_\poset{SP}) \leq \mathfrak b$ provable in ZFC? 
In a previous version of the paper, we tried to use the method implemented in Proposition \ref{2.4} in order to prove $\text{add}(\mathcal{I}_\poset{P}) \leq \mathfrak b$, for $\poset{P} \in \{ \spl,\fspl \}$. We deeply thank Otmar Spinas to find a gap in the argument and write us in a private communication. The point where our proof specifically works in this case is that the $\poset{P}$-generic tree $q$ is covered by  any $\poset{P}$-open dense from the ground model $V$, but when trying to use a similar argument for proving $\text{add}(\mathcal{I}_\poset{P}) \leq \mathfrak b$ one needs a finer argument to obtain $H(c)$ be caught by a tree $p$ in a given $\poset{P}$-open dense set in $N$.   
\end{remark}

\section{A brief digression on the Silver amoeba and Cohen reals}
\label{S3a}

Spinas \cite{Sp15} showed that the ideal of meager sets $\meager$ Tukey-embeds into the $\sigma$-ideal $\ideal{I}_\silver$ corresponding to Silver forcing, in symbols $\meager \leq_T \ideal{I}_\silver$. This result is surprising because it is in sharp contrast with the other popular non-ccc tree forcings: see \cite{LShV93} and \cite{Lag14} for Sacks, \cite{Sp95} for Laver and Miller, and for Mathias is folklore. Spinas' brilliant proof idea essentially involves two key steps: 1) a quite technically demanding investigation of the Silver antichain number; 2) a coding by Hamming weights of a Cohen real \emph{inside} a Silver tree. 
This result concerning the existence of such a Tukey embedding is in parallel with the fact that any amoeba for Silver necessarily adds Cohen reals. In fact, given any absolute amoeba $\mathbb{A}\silver$ for Silver and assume $\mathbb{A}\silver$ is proper. Then, a countable support iteration of length $\omega_2$, blows up  $\add(\mathcal{I_\silver})$. Now using Spinas' result of the Tukey-reducibility one gets in the generic extension $\aleph_1<\add(\mathcal{I_\silver})\leq \add(\mathcal{M})$ and thus one deduces that Cohen reals must have been added during the iteration. However, it remains unclear, whether a single step of the amoeba $\mathbb{A}\silver$ necessarily adds a Cohen. For instance, the following situation may occur: First $\mathbb{A}\silver$ adds half a Cohen but no Cohen reals and second $\mathbb{A}\silver * \mathbb{A}\silver$ adds a Cohen real. Forcings with these two properties exist (see \cite[Theorem 1.3]{ZAPLETAL201431}).

In this section we give a direct proof that any absolute amoeba for Silver adds Cohen reals. We still use the coding from part 2), but replace the argument about Silver antichains with an alternative method based on Cohen forcing as in the proof of Proposition \ref{2.4}.
A similar reasoning as in Remark \ref{remark 3.8} prevents us to get $\meager \leq_T \ideal{I}_\silver$ as well.

We do not go through the details of the coding (by Hamming weights) and we refer the reader to \cite[Coding Lemma 2 and Thinning-Out Lemma 3]{Sp15} for the proofs. 
We only recall from \cite{Sp15} the definition and the main result we need in our proof below.

\begin{definition}
For $n,k < \omega$ we define $d(n, k) \in 2$ by letting $d(n, k) = 0$ iff the
unique $j < \omega$ such that $n \in [j\cdot 2 k,(j+1)\cdot 2 k )$ is even. Let $e(n)$ be the minimal
$k < \omega$ with $2 k > n$. Let $c(n) = \langle d(n, k) \such k < \omega \rangle$ and $c^*(n) = c(n) \restric e(n)$.

For $\sigma \in 2^{<\omega}$ the \emph{Hamming weight of $\sigma$} is defined as $HW(\sigma):= |\{ i < |\sigma|: \sigma(i)=1 \}|$.
\end{definition}

As usual we denote with $\cohen=(2^{<\omega},\subseteq)$ the Cohen forcing. 

\begin{lemma}\emph{(\cite[Thinning-Out-Lemma 3]{Sp15})}
Given $\{ D_j : j < \omega \}$ a family of open dense sets $D_j \subseteq \cohen$ and $p \in \silver$, there exists $q \in \silver$ such that $q \leq p$ and for every $n < \omega$
and $\sigma \in \splitting_{n+1}(q)$, if $m = |\tau |$ for any $\tau \in \splitting_{n}(q)$ and $k = HW (\sigma)$, then
$w^\conc c^* (k) \in D_j$ for every $w \in 2^{\leq m}$ and $j \leq n$.
\end{lemma}
 
The property given in the Thinning-Out-Lemma is not only dense, but even open, i.e., for every $p \in \silver$ there exists $q \leq p$ such that every $q' \leq q$ satisfies the property of the Thinning-Out-Lemma (see \cite[Remark 3]{Sp15}).

\begin{proposition}
Let $\silver$ be the Silver forcing and let $V \subseteq N$ be models of ZFC. If
\[
N \models \text{``There is an absolute $\poset{V}$-generic tree over $V$''}
\]
Then
\[
N \models \text{``There is a Cohen real over $V$''}. 
\]
Hence any absolute amoeba forcing for $\silver$ adds a Cohen real. 
\end{proposition}

\begin{proof}
The proof follows the line of that one for Proposition \ref{2.4}.
First of all, given $E \subseteq \cohen$ open dense, we apply Spinas' Thinning-Out-Lemma with $\{ D_j : j < \omega \}$ such that $D_j=E$, for all $j \in \omega$,  in order to get a $\silver$-open dense set $D_E \subseteq \silver$ such that for every $q \in D_E$
\begin{equation} \label{eq2}
(\forall n < \omega)( \forall w  \in 2^{\leq m_n})( w^\conc \xi_n\in   E), 
\end{equation}
where $\xi_n:= c^*(HW(\sigma_n))$ with $\sigma_n$ is the leftmost sequence in $\splitting_{n+1}(q)$, and $m_n \in \omega$ is the length of any $\tau \in \splitting_{n}(q)$.

 Let $q \in N$ be an absolute $\silver$-generic tree over the ground model $V$. As in Definition \ref{1.1} (o), let $\bar{H}: 2^{<\omega} \rightarrow \splitting(q)$ be the $\trianglelefteq$-preserving function  and $H: 2^\omega \rightarrow [q]$ its natural extension. Let $c$ be a Cohen real over $V$. Like in the proof of Proposition \ref{2.4}, we conclude that for every $\silver$-open dense set $D\in V$ there exists $p \in D$ such that $N[c] \models H(c) \in [p]$, and so there exists $\sigma \in \cohen$ such that $\sigma \force H(c) \in [p]$. As before we get the following claim.

\begin{claim}\label{q sigma}
For every $\silver$-open dense set $D \in V$ there exists $\sigma \in \cohen$ and $p\in D$ such that  $q\restric  \bar{H}(\sigma) \leq p$.
\end{claim}

So let $\{ q_k: k \in \omega \}$ enumerate all $q\restric \bar H (\sigma)$'s, and let $m^k_n$, $\xi^k_n$ be associated with $q_k$ as in (\ref{eq2}) above. Then put \[ z:= {\xi^0_{n_0}}^\conc {\xi^1_{n_1}}^\conc \dots^\conc {\xi^k_{n_k}}^\conc \dots, \] where the $n_k$'s are chosen recursively as follows:
$n_0=0$ and for $k \geq 1$, $n_{k}$ is such that $|{\xi^0_{n_0}}^\conc {\xi^1_{n_1}}^\conc \dots^\conc {\xi^{k-1}_{n_{k-1}}}| \leq m^{k}_{n_{k}}$. 

We aim at showing that $z$ is Cohen over $V$. Let $E \subseteq \cohen$ be open dense in $V$. Let $D_E \subseteq \silver$ be a $\silver$-open dense set in $V$ such that every $p \in D_E$ satisfies  (\ref{eq2}).  By Claim \ref{q sigma} there exists $k \in \omega$ such that $q_k \leq p$, for some $p \in D_E$. By construction, for every $w \in 2^{\leq m^k_{n_k}}$ we have $w^\conc \xi^{k}_{n_{k}} \in E$. Since $|{\xi^0_{n_0}}^\conc {\xi^1_{n_1}}^\conc \dots^\conc {\xi^{k-1}_{n_{k-1}}}| \leq m^{k}_{n_{k}}$, it follows ${\xi^0_{n_0}}^\conc {\xi^1_{n_1}}^\conc \dots^\conc {\xi^{k-1}_{n_{k-1}}}^\conc {\xi^{k}_{n_k}} \in E$. 

Hence, for every $E \in V$ open dense of $\cohen$, there exists $n \in \omega$ such that $z\restric n \in E$, which means $z$ is Cohen over $V$.

\end{proof}

\section{Fatness versus measure}

Let $\mu$ denote the standard measure on $2^\omega$. Then, the Random forcing $\mathbb{B}$ consists of all perfect trees $p\subseteq 2^{<\omega}$ with positive measure, ordered by inclusion. We denote the $\sigma$-ideal of measure zero sets with $\mathcal{N}$.\\
In this section we compare $\fspl$ with $\mathbb{B}$. We show that each random condition is modulo a measure zero set equal to a fat splitting tree, but the converse does not hold. In fact, we show that the set of fat splitting trees with measure zero is dense in $\fspl$. We conclude this section by scrutinizing the differences between the two corresponding $\sigma$-ideals $\mathcal{I}_\fspl$ and $\mathcal{N}$. 

\begin{lemma} \label{fspl subset random}
$\mathbb{B} \cap \fspl $ is dense in $\mathbb{B}$.
\end{lemma}
\begin{proof}
We call a level $n$ of a tree $p\subseteq 2^{<\omega}$ \emph{nowhere splitting} if $|\Lev_n(p)|=|\Lev_{n+1} (p)|$, i.e., if there is no splitting node $s\in \Lev_n (p)$. Let $p\in \mathbb{B}$ be given and put $N:=\{ t \in p \such \mu ([p\restric t])=0\}$. This is a countable set and therefore $\mu(\bigcup_{t\in N}[p\restric t])=0$. So, $q:= p\setminus N$ is still a perfect tree with positive measure and the additional property that for each $t\in q$ $(\mu([q\restric t])>0)$. We claim that $q$ is a fat splitting tree. To reach a contradiction assume that there is a node $t\in q$ with no corresponding $K_q(t)$, i.e., there are infinitely many $n\in \omega$ such that $\Lev_n(q\restric t )$ is nowhere splitting. We fix such a node $t\in q$ and an increasing enumeration  $\langle n_i \such i\in \omega \rangle$ of all nowhere splitting levels of $q\restric t$. It is enough to show that $\mu([q\restric t])\leq 2^{-(i+1)}$ holds for each $i\in \omega$. 
Fix $i\in \omega$. Since for each $j<i$ we know that $\Lev_{n_j}(q\restric t)$ is nowhere splitting, we get that $|\Lev_{n_i +1}(q\restric t)|\leq 2^{n_i-i}$. Therefore we get
\begin{align*}
\mu([q\restric t]) \leq \mu \left(\bigcup \{ [s] \such s \in \Lev_{n_i + 1} (q\restric t) \}\right) \leq 2^{n_i - i } \frac{1}{2^{n_i + 1 }} = 2^{-(i+1)}.
\end{align*}

\end{proof}

\begin{lemma}\label{large antichains}
Below any fat splitting tree $p\in \fspl$ we can find an antichain \\
 $\{ p^x \leq p \such  x\in 2^\omega \}$ such that $\mu([p^x])=0$ and $[p^x]\cap [p^y]=\emptyset$, whenever $x\neq y$.
\end{lemma}
\begin{proof}
  Let $p\in \fspl$ be given. By induction on $n\in \omega $ we construct a set of conditions $\{ p_n^s \such n\in \omega , s\in 2^{n} \}$
such that for any $n\in \omega$:
\begin{enumerate}
\item $p^{\langle \rangle}_0 = p ,$
\item $p^s_n \geq_{n+1} p^{s^\conc i}_{n+1}, i\in 2,$
\item $[p ^ {s^\conc 0} _ {n+1}]\cap [p^{s^\conc 1}_{n+1}]=\emptyset,$  
\item $\mu([p^{s^\conc i}_{n+1}])\leq \frac{1}{2}\mu([p^s_{n}])$.
\end{enumerate}
Let $n\in \omega$ and $s\in 2^n$ be given. We apply Lemma \ref{triple} to the condition $p^s_n$ to get two incompatible fat splitting trees $q_{n+1}^{s^\conc i},i\in 2$ satisfying conditions (2) and (3). We have to make sure that also (4) holds. Therefore we compare the two measures of $[q_{n+1}^{s^\conc i}],i\in 2$. W.l.o.g. assume $\mu([q_{n+1}^{s^\conc 0}])\leq \mu([q_{n+1}^{s^\conc 1}])$. Since the two trees $q_{n+1}^{s^\conc 0},q_{n+1}^{s^\conc 1}$ have disjoint bodies, we must have $\mu([q_{n+1}^{s^\conc 0}]) \leq \frac{1}{2}\mu([p^s_{n}])$. Thus, we can set $p_{n+1}^{s^\conc 0} := q_{n+1}^{s^\conc 0}$. Now via the same argument as above this time for $q_{n+1}^{s^\conc 1}$ instead of $p^s_n$, we get  a condition $p_{n+1}^{s^\conc 1}\leq_{n+1} q_{n+1}^{s^\conc 1} $ such that $\mu([p^{s^\conc 1}_{n+1}])\leq \frac{1}{2}\mu([q^{s^\conc 1}_{n+1}])$. This completes the construction. \\
Now for each $x\in 2^\omega$, we define $p^x:=\bigcap_{n\in \omega} p^{x\restric n}_{n}$. This is a fat splitting tree by (2) and has measure zero by condition (4). Condition (3) ensures that for two different $x\neq y\in 2^{\omega}$ we have $[p^x]\cap [p^y]=\emptyset$.   
\end{proof}

\begin{corollary}\label{fat dense set}
$\fspl \cap \mathcal{N}$ is dense in $\fspl$.
\end{corollary}

\begin{corollary}
$\mathcal{N} \setminus \mathcal{I}_\fspl \neq \emptyset$.
\end{corollary}
\begin{proof}
Any $p\in \fspl$ with measure zero is a witness for $[p] \in \mathcal{N} \setminus \mathcal{I}_\fspl. $
\end{proof}

\begin{proposition} 
Assume $\ensuremath{\mathsf{add}(\mathcal{N})=\mathfrak{c}}$. Then $\mathcal{I}_\fspl \setminus \mathcal{N} \neq \emptyset.$
\end{proposition}

\begin{proof}
The proof follows the idea in \cite[1.4]{brendle:stroll}.
Let $\langle p_\alpha\in \fspl  \such \alpha < \mathfrak{c} \rangle$ be an enumeration of all fat splitting trees with measure zero. By Corollary \ref{fat dense set} this is a dense set in $\fspl$. Now fix an enumeration $\langle q_\alpha \in \mathbb{B}  \such \alpha < \mathfrak{c} \rangle$ of the random forcing. Our aim is to construct a sequence $\langle x_\alpha \in 2^\omega \such \alpha < \mathfrak{c} \rangle$ such that
\begin{itemize}
\item[i)] $x_\alpha \not\in \bigcup_{\beta<\alpha} [p_\beta] $ 
\item[ii)] $ x_\alpha \in [q_\alpha]$.
\end{itemize}
We first check that we can indeed find such a sequence and then verify that the resulting set $X := \{x_\alpha \such \alpha <\mathfrak{c} \}$ witnesses $\mathcal{I}_\fspl \setminus \mathcal{N} \neq \emptyset$. So fix $\alpha<\mathfrak{c}$. Our assumption $\ensuremath{\mathsf{add}(\mathcal{N})=\mathfrak{c}}$ implies that $[q_\alpha] \setminus\bigcup_{\beta<\alpha} [p_\beta] $ has  still positive measure and therefore we can pick $x_\alpha$ satisfying conditions i) and ii). Condition ii) ensures that $X\not\in \mathcal{N}$ holds. So we are left to show that we can find for each $p\in \fspl$ a stronger condition $q\leq p$ such that $[q]$ and $X$ are disjoint. Therefore, fix $p\in \fspl$ and pick $\alpha<\mathfrak{c}$ with $p_\alpha\leq p$. Now by Lemma \ref{large antichains} there is an antichain $\{ p'_\beta \in \fspl \such \beta<\mathfrak{c}  \}$  below $p_\alpha$ satisfying $[p'_\beta] \cap [p'_\gamma ] = \emptyset $, whenever $\beta\neq \gamma$. Condition i) implies that $|[p_\alpha] \cap X| < \mathfrak{c}$. In particular, we can find some $\beta$ with $[p'_\beta] \cap X = \emptyset$.
\end{proof}

\begin{question}
Does $\mathcal{I}_\fspl\setminus\mathcal{N}\neq\emptyset$ hold without the assumption $\ensuremath{\mathsf{add}(\mathcal{N})=\mathfrak{c}}$?
\end{question}

\section{A possible difference between $\spl$ and $\fspl$}
\label{s6}

Our results of this section apply just to fat splitting tree forcing.
The analogous facts for splitting tree forcings are open.

\begin{definition}\hfill
  \begin{enumerate}
    \item $\la S_n \such n < \omega\ra$ is called an \emph{$f$-slalom} if
$S_n \subseteq [\omega]^{f(n)}$.
\item 
A forcing $\bP$ has \emph{the Sacks property} if for any
$f \colon \omega \to \omega \setminus\{\emptyset\}$ such that $\lim_n f(n)= \infty$ and any $\bP$-name $\tau$ for a real and any condition $p$ there is an
$f$-slalom $\la S_n \such n < \omega \ra$ and there is $q \leq p$ such that
\[q \Vdash (\forall^\infty n)(\tau(n) \in S_n).\]
  \end{enumerate}
\end{definition}
\begin{lemma}\label{unreachable reals}
  For any function $f \colon \omega \to \omega \setminus\{\emptyset\}$
  there is a $\fspl$-name $\tau$ such that for any $f$-slalom
  $\la S_n \such n < \omega\ra \in V$
  \[\fspl \Vdash (\exists^\infty n) (\tau(n) \not\in S_n).\] 

\end{lemma}
\begin{proof}
Fix $f \in \omega^\omega \cap V$. We aim at finding $\dot{h} \in \omega^\omega \cap V^\fspl$ so that for every slalom $S \in ([\omega]^{<\omega})^\omega \cap V$, with $|S(n)|\leq f(n)$, we have $h$ is not captured by $S$, i.e., there is $n \in \omega$ such that $h(n) \notin S(n)$. 
We let $\dot{x}$ be the name for the $\fspl$-generic real,
i.e. the union $\bigcup\{\stem(p) \such p \in G\}$.

Let $\codes \colon \{2^I \such I \subseteq \omega, I \mbox{ finite}\} \to \omega$
be an injective function.

We fix a partition $\{ I_n \such n \in \omega \}$ of $\omega$, where each $I_n$ is an interval, $\max I_n < \min I_{n+1}$, and $|I_n| > f(n)$.
We define
\[
h:= \langle \codes(x\restric I_n) \such  n \in \omega \rangle.
\] 
We aim to show that $h$ cannot be captured by any $f$-slalom in the ground model.
So fix an $f$-slalom $S \in V$ and a condition $p \in \fspl$. It is enough to find $n \in \omega$, $q \leq p$ such that $q \force h(n) \notin S(n)$. Pick $n \in \omega$ such that $\min(I_n) > K_p(\stem(p))$.
Then for every $j \in I_n$, there is $t \in \splitting(p)$ such that $|t|=j$. Hence on level $\max(I_n)$ of $p$
we have at least $|I_n| > f(n)$ nodes that are
pairwise different in $I_n$. Let $\{t_k \such k \in |I_n|\}$ enumerate
the first $|I_n|$ of them.
Then 
\[
p\restric t_k \force h(n)= \codes (x \restric I_n) = \codes (t_k \restric I_n),
\]
Since the $t_k\restric I_n$ are pairwise different and $\codes$ is injective there are at least $|I_n|$ possibilities for $p$ to decide $h(n)$ and because $|S(n)| \leq f(n) < |I_n|$ there is some $k\in |I_n|$ such that
\[
p \restric t_k \force \codes(x\restric I_n) = \codes (t_k \restric I_n) = h(n) \notin S(n).
\]
So, $q:=p \restric t_k \leq p$ is the condition  with the desired property.
\end{proof}
\begin{corollary}
$\fspl$ does not satisfy the Sacks property.
\end{corollary}

\begin{question} Does $\spl$ have the Sacks property?
  \end{question}
  
\begin{lemma}
$\mathcal{I}_\fspl \setminus \mathcal{I}_\spl \neq \emptyset$.
\end{lemma}
\begin{proof}
We can take any splitting tree $p\in \spl$ such that for infinitely many $n\in \omega$ there is no splitting node in $\level_n(p)$. Then $[p]\in \mathcal{I}_\fspl \setminus \mathcal{I}_\spl.$
\end{proof}
\begin{question}
Does $\mathcal{I}_\spl \setminus \mathcal{I}_\fspl \neq \emptyset$ hold?
\end{question}

\section{Splitting-measurability}
\label{S4}

Here we investigate the complexity of $\bP$-measurable sets.
The first three results hold for fat splitting trees and splitting trees, while the proof of Theorem \ref{separation-fsp-sacks}  specifically uses Lemma \ref{unreachable reals}, which only applies to $\fspl$.

Note that if $X$ is weakly $\fspl$-measurable, then it is also weakly $\spl$-measurable.

\begin{proposition} \label{Baire-split}
For every set $X\subseteq 2^\omega$ we have
\begin{enumerate}
\item $X$ has the Baire property implies  $X$ is weakly $\fspl$-measurable.
\item $X$ is Lebesgue-measurable implies  $X$ is weakly $\fspl$-measurable.
\end{enumerate}
\end{proposition}

\begin{proof}
(1) Recall that $X$ has the Baire property implies that $X$ either is meager or  there is $s\in2^{<\omega} $ such that $X\cap[s] $ is comeager in $[s]$. Therefore, it is enough to show that any for comeager set $D$ there exists $p\in \fspl$ such that $[p] \subseteq D$. So fix a comeager set $D$ and let $\{ D_n: n \in \omega \}$ be a $\subseteq$-decreasing family of open dense subsets such that $\bigcap D_n \subseteq D$. We aim at finding $p \in \fspl$ such that $[p] \subseteq  \bigcap_{n\in \omega} D_n$. We will do so by constructing an $\subseteq$-increasing family of finite trees $\{F_n\subseteq 2^{<\omega} \such n\in \omega\}$ and taking $p:=\bigcup_n F_n$. 
Consider the following recursive construction:
\begin{enumerate}
\item[Step $0$.] For $i\in 2$, pick $t_i \trianglerighteq \langle i \rangle$ such that $[t_i]\subseteq  D_0$, and $|t_1| > |t_0|$.  For every $\langle 0 \rangle \trianglelefteq s \triangleleft t_0$, let $s'_j:= s^\conc j$, where $j \in \{ 0,1\}$ is chosen so that $s^\conc j \ntrianglelefteq t_0$. Then pick $t'_j \trianglerighteq s'_j$ such that $[t'_j] \subseteq D_0$ and $|t'_j| > |t_1|$. Let $T_0$ be the set containing $t_0,t_1$ and all such ${t_j}'$s and put $N:= \max \{ |t| \such t\in T_0 \}$. Finally consider the set
\[
F_0 := \{ s\in 2^{\leq N} \such \exists t\in T_0 (s\trianglelefteq t \vee t\trianglelefteq s)  \}.
\]
By construction, $[F_0]\subseteq D_0$ and 
\[
\forall n < N \exists s \in \level_n(F_0) (s \text{ is splitting}).
\]
\item[ $n+1$.] Assume we already constructed $F_n$. Then, for every $t \in \ter(F_n)$, we can repeat the construction described at Step $0$ starting with $t_i \trianglerighteq t^\conc i,(i\in 2)$ in order to construct $F^t_{n+1}$ satisfying the following properties:
\begin{itemize}
\item $\forall t' \in F^t_{n+1}$ $(t \trianglelefteq t' \vee t' \trianglelefteq t)$,
\item $[F^t_{n+1}]\subseteq  D_{n+1}$.
\end{itemize} 
Let $T_{n+1}:= \bigcup \{ \ter(F^t_{n+1}) \such t\in \ter(F_n) \}$ and  $N:= \max\{|t| \such  t\in T_{n+1} \} $. We define
\[
F_{n+1}:= \{s\in 2^{\leq N} \such \exists t\in T_{n+1} (s\trianglelefteq t \vee t\trianglelefteq s)\}.  \]
\end{enumerate}
Finally let $p := \bigcup_{n \in \omega} F_n$. By construction $[p] \subseteq \bigcap_{n \in \omega} D_n$ and $p \in \fspl$. \\

(2) Recall that  $X$ Lebesgue measurable implies that  there is $p\in \mathbb{B}$ such that $[p]\subseteq X$ or $[p]\cap X= \emptyset$. By Lemma \ref{fspl subset random} there is $p'\subseteq p$ such that $p'\in \fspl$.

\end{proof}

\begin{proposition}\label{Brendle_Loewe_Halbeisen}
Let $G$ be $\cohen_{\omega_1}$-generic over $V$. Then
\[
V[G] \models \text{All $\text{On}^\omega$-definable set are weakly-$\fspl$-measurable}.
\]  
\end{proposition}

In order to prove Proposition~\ref{Brendle_Loewe_Halbeisen} we need the following result.

\begin{lemma} \label{lemma:splitting-cohen}
Cohen forcing adds a $\fspl$-tree consisting of Cohen branches, i.e. $\cohen$ adds a tree $q \in \fspl$ such that 
\[
\force_\cohen \forall x \in [q] (x \text{ is Cohen over $V$}).
\]
\end{lemma}

\begin{proof}
Consider the forcing $\poset{P}$ defined as follows: $F \in \poset{P}$ iff 
\begin{enumerate}
\item $F \subseteq 2^{<\omega}$ is a finite tree,
\item $\forall s,t \in \ter(F) (|s|=|t|)$.
\end{enumerate}
The partial order on $\poset{P}$ is given by:
\begin{align*}
F'\leq F \leftrightarrow & F\subseteq F' \wedge \forall t\in F' \setminus F \exists s\in \ter(F) (s\unlhd t)\\
 &\wedge \forall s\in \ter(F) \forall n\in [\height(F),\height(F')) \exists t\in \level_n((p \restric s) \restric F') (t\concat 0 , t\concat 1 \in F').
\end{align*}
Given two conditions $F_1 \leq F_0$ the partial order $\leq $ satisfies the following:  For $p_i:=\{t\in 2^{<\omega} \such \exists s \in \ter{(F_i)} (t\unlhd s \vee s \unlhd t) \}$,  we have $p_0,p_1\in \fspl$ and $p_1\leq_{n} p_0$, where $n$ is the number of splitting levels of $F_0$. Specifically, taking a $\poset{P}$-generic filter $G$ and defining $p_G:= \bigcup G$, we also get that $p_G$ is a fat splitting tree (in the generic extension).\\
 Note that $\poset{P}$ is a countable atomless forcing order and so equivalent to $\cohen$. In fact, to see that $\poset{P}$ is atomless let $F\in \poset{P}$ be given. We have to find two incompatible extensions $F_0,F_1$ of $F$. Let $n = \height(F)$. We construct $F_i,i\in 2$ in three steps:
\begin{align*}
& F'_i := F \cup \{t\in 2^{n+1} \such \exists s\in \ter(F) (s\unlhd t) \}, \\
& F''_i := F' \cup \{t\in 2^{n+2} \such \exists s\in \ter(F') (s\unlhd t \wedge (s(n)=0 \rightarrow t(n+1)=i)) \} \\
&F_i := F'' \cup \{t\in 2^{n+3} \such \exists s\in \ter(F'') (s\unlhd t \wedge (s(n)=1 \rightarrow t(n+2)=i)) \}. 
\end{align*}
By the above we make sure that the terminal nodes of $F_0$ and $F_1$ are disjoint and in particular they are incompatible in $\poset{P}$. This proves that $\poset{P}$ is atomless.\\
Let $p_G:= \bigcup G$, where $G$ is $\poset{P}$-generic over $V$. It is left to show that every branch in $[p_G]$ is Cohen. 

So let $D$ be an open dense subset of $\cohen$ and $F \in \poset{P}$. It is enough to find $F' \leq F$ such that every $t \in \ter(F')$ is a member of $D$.

Therefore fix arbitrarily $t \in \ter(F)$ and consider the following construction.
Pick $t_0 \unrhd t^\conc 0$ such that $t_0 \in D$ and put 
\[
F_0(t):= F \cup \{s \in 2^{<\omega}: t^\conc 0 \unlhd s \unlhd t_0  \} \cup \{ s \in 2^{<\omega}: t^\conc 1 \unlhd s \land |s| \leq |t_0| \}.
\]
Then for every $s \in \ter(F_0(t))$ with $s \unrhd t^\conc 1$, pick $t_s \unrhd s$ such that $t_s \in D$. Note that since $D$ is open dense and we only deal with finitely many $s \in \ter(F_0(t))$, we can pick the $t_s$'s with the same length, say $N_t$. We then define
\[
F_1(t):= F \cup \{ r \in 2^{<\omega} : t_0 \unlhd r \land |r| \leq N_t  \} \cup  \{ r \in 2^{<\omega} : \exists s \in \ter(F_0(t)) ( t^\conc  1 \unlhd s \unlhd r \unlhd t_s) \}.
\]

Now we let $F'':= \bigcup \{F_1(t): t \in \ter(F) \}$. We are almost done, we only have to make sure that all terminal nodes are of the same length. Therefore let $N:= \max\{N_t \such t\in \ter(F) \}$ and define $F':=\{ s\in 2^{\leq N} \such \exists r\in \ter(F'') (r\unlhd s) \}$. 
By construction, $F' \leq F$ and $\ter(F')\subseteq D$.

\end{proof}

\begin{proof}[Proof of Proposition \ref{Brendle_Loewe_Halbeisen}]
  The argument follows the line of the proof of \cite[Proposition 3.7]{BLH2005}.
Let $G$ be $\bC_{\omega_1}$-generic over $V$.  In $V[G]$, let $X$ be an $\text{On}^\omega$-definable set of reals, i.e. $X:= \{x \in 2^\omega: \varphi(x,v)   \}$ for a parameter $v \in \text{On}^\omega$.  We aim to find $p \in \fspl$ such that $[p] \subseteq X$ or $[p] \cap X=\emptyset$. 

First note that we can absorb $v$ in the ground model, i.e, we can find $\alpha<\omega_1$ such that $v \in V[G \restric \alpha]$.
We let $c= G(\alpha)$ be the next Cohen real and we write $\bC$ for
the $\alpha$-component of $\bC_{\omega_1}$.

There are $s_0,s_1 \in \cohen$
  such that $s_0 \leq b_0 =  \big \llbracket \llbracket (\varphi(c,v)) \rrbracket_{\cohen_{\alpha}} = \mathbf{0}  \big \rrbracket_\cohen$ and
    $s_1 \leq b_1 =  \big \llbracket \llbracket
    \varphi(c,v))\rrbracket_{\cohen_{\alpha}}= \mathbf {1}  \big \rrbracket_\cohen$.
We can then find $p \in \fspl$ as in Lemma \ref{lemma:splitting-cohen} such that $[p] \subseteq b_0$ or $[p] \subseteq b_1$ and every $x \in [p]$ is Cohen over $V[G \restric \alpha]$. We claim that $p$ satisfies the required property.
\begin{itemize}
\item Case $[p] \subseteq b_1$: note every $x \in [p]$ is Cohen over $V[G \restric \alpha]$, and so $V[G\restric \alpha][x] \models \llbracket \varphi(x,v)=\mathbf{1} \rrbracket_{\cohen_{\omega_1}}$. Hence $V[G] \models \forall x \in [p](\varphi(x,v))$, which means $V[G] \models [p] \subseteq X$.  
\item Case $[p] \subseteq b_0$: we argue analogously and get $V[G] \models \forall x \in [p](\neg \varphi(x,v))$, which means $V[G] \models [p] \cap X = \emptyset$.
\end{itemize}

\end{proof}

\begin{theorem} \label{separation-fsp-sacks}
Assume there exists $\kappa$ inaccessible. There is a model for \textsc{ZF$+$DC} where all sets are $\silver$-measurable (and so $\sacks$-measurable as well, by Remark \ref{Remark2}) but there is a set which is not $\fspl$-measurable. 
\end{theorem}

\begin{proof}
The key idea is to get a complete Boolean algebra $B$ and a $B$-name $Y$ for a set of elements in $2^\omega$ such that in the corresponding extension $V[G]$, for a $B$-generic filter $G$, the following hold:
\begin{enumerate}
\item every subset of $2^\omega$ in $L(\real,Y)$ is $\silver$-measurable
\item $Y$ is not $\fspl$-measurable.
\end{enumerate}
Hence, we obtain that in $L(\real,Y)^{V[G]}$ every subset of $2^\omega$ is $\silver$-measurable, but there is a set which is not $\fspl$-measurable. 

We start with a definition.
\begin{definition}
  A complete Boolean algebra $B$ is 
  \emph{$(\silver,Y)$-homogeneous}, if for every Silver algebras $B_0,B_1 \lessdot B$ and any isomorphism $\phi \colon  B_0 \rightarrow B_1$, there exists
  an automorphism $\phi^+ \supseteq \phi$ of $B$ such that
  \[\force_B \phi^+[Y]=Y.
  \]
\end{definition}

  To construct a $(\silver,Y)$-homogenous Boolean algebra, we use Shelah's amalgamation. We start by sketching out such Shelah's procedure.

One \emph{basic amalgamation step} consists of $\omega$ substeps and looks as follows.
Given a Boolean algebra $B= \amal^0(B,\phi)$, two complete subalgebras $B_0, B_1 \lessdot B$ and an isomorphism $\phi \colon B_0 \rightarrow B_1$, the amalgamation process provides us with the pair $(\amal(B,\phi),\phi^1)$ such that $B \lessdot \amal(B,\phi)$, there are two isomorphic copies $e_0[B]$, $e_1[B] \lessdot \amal(B,\phi)$ of $B$  and $\phi^1 \supseteq \phi$ such that $\phi^1 \colon e_0[B] \rightarrow e_1[B]$ is  an isomorphism. Such a procedure can  be repeated, now with $e_i[B]$, $\phi^1$ and $\amal(B,\phi)=\amal^1(B,\phi) $, and thus we get $\amal^2(B,\phi)$ and $\phi^2$.
At the limit stage $\omega$ we take the smallest complete superalgebra $\amal^\omega(B,\phi)$ of $\amal^n(B,\phi)$, $n < \omega$, as described in details in \cite[p.10]{JR93}, or more briefly in \cite[p.736]{Lag14-2}.  We let $\phi^\omega = \bigcup \phi^n$. The corresponding automorphism $\phi^\omega: \amal^\omega(B,\phi) \rightarrow \amal^\omega(B,\phi)$ fulfills $\phi^\omega \supseteq \phi$.

We construct $B$ by a recursive construction of length $\kappa$ for a strongly inaccessible cardinal $\kappa$.   We partition $\kappa$ into four cofinal sets $S_i$, $i = 0, 1, 2,  3$.
By induction on $\alpha\leq \kappa$ we choose $B_\alpha$ and $Y_\alpha$.
We start with $B_0 = \{0\}$, $Y_0 = \emptyset$.

\begin{myrules}
\item[(a)]
For $\alpha \in S_0$, we let 
\begin{equation*}
  \begin{split}
    B_{\alpha+1} &= B_\alpha \ast \amoeba(\silver),\\
    B_{\alpha+1} & \Vdash \dot{Y}_{\alpha+1} = \dot{Y}_\alpha.
  \end{split}
\end{equation*}
The amoeba for silver forcing is denoted by $\amoeba(\silver)$. It is a forcing for adding a Silver tree consisting of $\silver$-generic branches, and it is used in the variant of Solovay's Lemma in order to obtain $\silver$-measurability.  
\item[(b)] For $\alpha \in S_1$,
  we use a standard book-keeping argument
  to hand us down all situations of the following kind:
$B_\alpha \lessdot B' \lessdot B_\kappa$ and $B_\alpha \lessdot B'' \lessdot B_\kappa$ are such that
$B_\alpha$ forces $(B':B_\alpha) \approx (B'':B_\alpha) \approx \silver$
  and $\phi_\alpha \colon  B' \rightarrow B''$ an isomorphism s.t.\  $\phi_\alpha \upharpoonright B_\alpha= \text{Id}_{B_\alpha}$. So suppose that $\phi_\alpha \colon B' \to B''$, where
  $B'$ and $B''$ are two Silver algebras of $B_\alpha$, is handed down by the book-keeping. Then we let 
  \begin{equation*}
    \begin{split}
      B_{\alpha+1} & = {\rm Am}^\omega(B_\alpha, \phi_\alpha),\\
      B_{\alpha+1} & \Vdash \dot{Y}_{\alpha+1} = \dot{Y}_{\alpha} \cup \{
\phi^j_{\alpha}(\dot{y}), \phi^{-j}_{\alpha}(\dot{y}): \dot{y} \in
Y_{\alpha}, j \in \omega \}.
    \end{split}
  \end{equation*}

\item[(c)] For $\alpha \in S_2$, we let
 \[
B_{\alpha + 1} :=
\algebra{B}_\alpha \ast \prod_{p\in \fspl^{V_{B_\alpha}}}  \dot{\fspl_p}, 
\]
where $\fspl_p:= \{q \in \fspl \such q\leq p \}$ and 
\[
B_{\alpha+1} \Vdash \dot{Y}_{\alpha +1} := \dot{Y}_\alpha \cup \{
\dot{y}_p \such p \in \fspl^{V_{B_\alpha}} \},
\]
where $\dot{y}_p$ is the standard name for the $\fspl_p$-generic real over
$\model{V}^{B_\alpha}$.

\item[(d)] For $\alpha \in S_3$, we let
\[
\algebra{B}_{\alpha + 1} :=
\algebra{B}_\alpha \ast \levy(\omega, \alpha),
\] and 
$B_{\alpha+1} \Vdash \dot{Y}_{\alpha +1} := \dot{Y}_\alpha$.
Here $\levy(\omega, \alpha)$ is the L\'evy collapse of $\alpha$ to $\omega$, i.e., the set of $p \colon n \to \alpha$, $n \in \omega$,
ordered by end-extension.
 \item[(e)]  Finally, for any limit ordinal $\lambda\leq \kappa$, we let $B_\lambda = \lim_{\alpha < \lambda} B_\alpha$ and  $B_\lambda \Vdash \dot{Y}_\lambda = \bigcup_{\alpha < \lambda} \dot{Y}_\alpha$.
\end{myrules}

We let $B= B_\kappa$ and $Y = Y_\kappa$ and show that they are as in (1) and (2).

When amalgamating over Silver forcing (as in the construction \cite[pp. 740-741]{Lag14-2}) in order to get $\silver$-measurability, we need to isolate a particular property shared by the $\fspl$-generic reals (namely \emph{unreachability}, i.e., reals which are not captured by any ground model slalom (introduced in \cite[Def. 12]{Lag14-2})), which is both preserved under amalgamation (\cite[Lemma 15]{Lag14-2}) and under iteration with Silver forcing (\cite[Lemma 16]{Lag14-2}).

We recall here the definition of \emph{unreachability} and some main remarks for the reader convenience. 
\begin{definition}
  \hfill
  \begin{itemize}
\item $\Gamma_k = \{ \sigma \in {\rm HF}^\omega: \forall n \in (|\sigma(n)| \leq 2^{kn})   \}  \}$ and $\Gamma= \bigcup_{k \in \omega} \Gamma_k$, where ${\rm HF}$ denotes the hereditary finite sets;
\item $g(n)=2^{n\cdot n}$;
\item $\{J_n: n \in \omega \}$ is defined via $J_0= \{ 0 \}$ and $J_{n+1}= \Big [\sum_{j \leq n} g(j), \sum_{j \leq n+1} g(j) \Big )$, for every $n \in \omega$;
\item Given $x \in \cantor$, define $h_x(n)= x \restric J_n$.
\item One says that $z \in \cantor$ is \enfa{unreachable over $\model{V}$} if
\[
\forall \sigma \in \Gamma \cap \model{V} \exists n \in \omega (h_z(n) \notin \sigma(n)).
\]
\end{itemize}
\end{definition}
By Lemma  \ref{unreachable reals}, applied for each $k \geq 1$ to $f_k(n) = 2^{kn}$,
with our modification $J_n$ instead of $I_{k,n}$ we have for the generic $\dot{y}_p$ of $\fspl_p$: For each $p\in \fspl$ the real
$\dot{y}_p$ is unreachable over $V$.

The proof of Theorem \ref{separation-fsp-sacks} is concluded as follows:

Let $G$ be a $\algebra{B}_\kappa$-generic filter over
$\model{V}$. Then
\[
\model{V}[G] \models \text{`` $Y$ is not $\fspl$-measurable''}.
\]

One has to prove that for every $p \in \fspl$, both $Y \cap [p] \neq \emptyset$ and $[p] \not \subseteq Y$.

The proof follows the line of \cite[Lemma 28]{Lag14-2}. More specifically, to prove the part $Y \cap [p] \neq \emptyset$ it is enough to use item (c) of the construction, by choosing $\alpha < \kappa$ sufficiently large so that $p \in V[G\restric \alpha]$ (possible by $\kappa$-cc) and then picking a $\fspl_p$-generic real over $V[G \restric \alpha]$, call it $y \in Y_{\alpha+1} \cap [p]$. The part $[p] \not \subseteq Y$ follows from the fact that if $p \in V[G \restric \alpha]$ then any new real added at stages $\beta>\alpha$ in $[p]$ cannot be in $Y$; the elements that enter $Y$ under clause (b) come from former stages and hence are not identical to the new real unless there were identical elements already in a former stratum of $Y$, and the elements entering $Y$ under clause
(c) are not in $V[G \restric \alpha]$.
For a more detailed proof we refer to \cite[Lemma 29]{Lag14-2}. The argument is similar to the proof of \cite[Theorem 6.2]{JR93}, specifically in the part to show that the set $\Gamma$ cannot have the Baire property, where the property of ``being unreachable'' replaces the property of ``being unbounded". 

To see that every subset of the reals in $L(\omega^\omega,Y)$ is $\silver$-measurable, 
we use the fact that any isomorphism between copies of smaller
Silver algebras can be extended to an automorphism of $B$ that fixes $Y$. This is a slightly more complex variant of the usual homogeneity of Levy Collapse, providing us with a way to apply a variant of Solovay's Lemma (e.g., see \cite[Theorem 6.2.b]{JR93} for Lebesgue measurability or \cite[Lemma 24]{Lag14-2} for Silver measurability) in order to show that $(\silver,Y)$-homogeneity implies that all sets in $L(\real,Y)$ are Silver measurable.
\end{proof}


\end{document}